\documentclass[a4paper,12pt]{amsart}

\usepackage{amsmath,amssymb,amsfonts,
bbm}
\usepackage{graphicx,
color}

\usepackage{geometry}
\usepackage{amsmath}

\usepackage{float}
\usepackage{caption}
\newtheorem{theorem}{Theorem}[section]
\newtheorem{lemma}[theorem]{Lemma}
\newtheorem{corollary}[theorem]{Corollary}
\newtheorem{proposition}[theorem]{Proposition}

\theoremstyle{remark}
\newtheorem{remark}[theorem]{Remark}

\newtheorem*{remark*}{Remark}

\theoremstyle{definition}

\newcommand\realp{\mathop{Re}}
\newcommand\dH{\,d{\mathcal H}^1}

\def\bR{\mathbb{R}}
\def\bC{\mathbb{C}}

\newcommand\cB{\mathcal{B}}
\newcommand\cA{\mathcal{A}}
\newcommand\cF{\mathcal{F}}
\newcommand\cS{\mathcal{S}}

\newcommand\cH{\mathcal{H}}

\newcommand\cV{\mathcal{V}}
\newcommand\bN{\mathbb{N}}

\newcommand{\PhiOmega}[1]{\Phi_\Omega(#1)}
\newcommand{\PhiOm}{\Phi_\Omega}

\numberwithin{equation}{section}

\begin{document}
	
\title[The Faber-Krahn inequality for the STFT]{The Faber-Krahn inequality for the Short-time Fourier transform}
\author{Fabio Nicola}
\address{Dipartimento di Scienze Matematiche, Politecnico di Torino, Corso Duca degli Abruzzi 24, 10129 Torino, Italy.}
\email{fabio.nicola@polito.it}
\author{Paolo Tilli}
\address{Dipartimento di Scienze Matematiche, Politecnico di Torino, Corso Duca degli Abruzzi 24, 10129 Torino, Italy.}
\email{paolo.tilli@polito.it}

\begin{abstract}
In this paper we solve an open problem concerning the characterization
of those measurable sets $\Omega\subset \bR^{2d}$ that, among
all sets having a prescribed 
Lebesgue measure, can trap the largest possible energy fraction in time-frequency space, where the energy density of a generic function $f\in L^2(\bR^d)$ is defined in terms of its 
Short-time Fourier transform (STFT) $\cV f(x,\omega)$, with Gaussian window. More precisely,
given a measurable set $\Omega\subset\bR^{2d}$  having measure $s> 0$, we prove that
the quantity
\[
\Phi_\Omega=\max\Big\{\int_\Omega|\cV f(x,\omega)|^2\,dxd\omega: f\in L^2(\bR^d),\ \|f\|_{L^2}=1\Big\},
\]
is largest possible if and only if $\Omega$ is equivalent, up to a negligible set,
to a ball of measure $s$, and in this case we
characterize all functions $f$ that achieve equality.

This result leads to a sharp uncertainty principle for the ``essential support" of the STFT
(when $d=1$, this can be summarized by the optimal bound $\Phi_\Omega\leq 1-e^{-|\Omega|}$, with
equality if and only if $\Omega$ is a ball).

Our approach, using techniques from measure theory after suitably rephrasing the problem in the Fock 
space, also leads to a local version of Lieb's uncertainty inequality for the STFT in $L^p$ 
when   $p\in [2,\infty)$, as well as to $L^p$-concentration estimates
when $p\in [1,\infty)$, thus proving a related conjecture. In all cases we identify the corresponding extremals.
\end{abstract}
\maketitle

\section{Introduction}
The notion of energy concentration for a function $f\in L^2(\bR)$ in the time-frequency plane is an issue of great theoretical and practical interest and can be formalised in terms of time-frequency distributions such as the so-called Short-time Fourier transform (STFT), defined as
\[
\cV f(x,\omega)=
\int_\bR e^{-2\pi i y\omega} f(y)\varphi(x-y)dy, \qquad x,\omega\in\bR,
\]
where $\varphi$ is the ``Gaussian window''
\begin{equation}
\label{defvarphi}
 \varphi(x)=2^{1/4}e^{-\pi x^2}, \quad x\in\bR,
\end{equation} 
normalized in such way that $\|\varphi\|_{L^2}=1$.
It is well known that $\cV f$ is a complex-valued, real analytic, bounded function and $\cV:L^2(\bR)\to L^2(\bR^2)$ is an isometry (see \cite{folland-book,grochenig-book,mallat,tataru}).

It is customary to interpret $|\cV f(x,\omega)|^2$ as the time-frequency energy density of  $f$ (see \cite{grochenig-book,mallat}). Consequently, the fraction of energy captured by a measurable subset $\Omega\subseteq \bR^2$ of a function $f\in L^2(\bR)\setminus\{0\}$ will be given by the Rayleigh quotient (see \cite{abreu2016,abreu2017,daubechies,marceca})
\begin{equation}\label{defphiomegaf} 
\PhiOmega{f}:= \frac{\int_\Omega |\cV f(x,\omega)|^2\, dxd\omega}{\int_{\bR^2} |\cV f(x,\omega)|^2\, dxd\omega}=\frac{\langle \cV^\ast \mathbbm{1}_\Omega \cV f,f\rangle}{\|f\|^2_{L^2}}.
\end{equation}
The bounded, nonnegative and self-adjoint operator $\cV^\ast \mathbbm{1}_\Omega \cV$ on $L^2(\bR)$ is known in the literature under several names, e.g. localization, concentration, Anti-Wick or Toeplitz operator, as well as time-frequency or time-varying filter. Since its first appearance in the works by Berezin \cite{berezin} and Daubechies \cite{daubechies}, the applications of such operators have been manifold and the related literature  is enormous: we refer to the books \cite{berezin-book,wong} and the survey \cite{cordero2007}, and the references therein, for an account of the main results.
\par

Now, when $\Omega$ has finite measure, $\cV^\ast \mathbbm{1}_\Omega \cV$ is a compact (in fact, trace class) operator. Its norm  $\|\cV^\ast \mathbbm{1}_\Omega \cV \|_{{\mathcal L}(L^2)}$,
given by  the quantity
\[
\PhiOm:=\max_{f\in L^2(\bR)\setminus\{0\}} \PhiOmega{f}
=
\max_{f\in L^2(\bR)\setminus\{0\}}
\frac{\langle \cV^\ast \mathbbm{1}_\Omega \cV f,f\rangle}{\|f\|^2_{L^2}},
\] 
represents the maximum fraction of energy that can in principle be trapped by
$\Omega$ for any signal $f\in L^2(\bR)$, and explicit
upper bounds for $\PhiOm$ are of considerable interest. Indeed, 
the analysis of the spectrum of $\cV^\ast \mathbbm{1}_\Omega \cV$ was initiated in the seminal paper \cite{daubechies} for radially symmetric $\Omega$, in which case the operator is diagonal in the basis of Hermite functions --and conversely \cite{abreu2012} if an Hermite function is an eigenfunction and $\Omega$ is simply connected then $\Omega$ is a ball centered at $0$-- and the asymptotics of the eigenvalues  (Weyl's law), in connection with the measure of $\Omega$, has been studied by many authors; again the literature is very large and we address the interested reader to the contributions \cite{abreu2016,abreu2017,demari,marceca,oldfield} and the references therein. 


The study of the time-frequency concentration of functions, 
in relation to uncertainty principles and under certain additional
constraints (e.g. on subsets of prescribed measure in phase space, 
or under limited bandwidth etc.) 
has a long history which,
as recognized by Landau and Pollak \cite{landau1961},
 dates back at least to Fuchs \cite{fuchs}, and its relevance
both to theory and applications has been well known 
 since the seminal works by
Landau-Pollack-Slepian, see e.g. \cite{folland,landau1985,slepian1983},
and other relevant contributions such as those of
Cowling and Price \cite{cowling}, Donoho and Stark \cite{donoho1989}, and  Daubechies \cite{daubechies}. 

However, in spite of the abundance
of deep and unexpected results related to this circle of ideas (see e.g.
the visionary work by Fefferman \cite{fefferman}) the question of characterizing
the subsets $\Omega\subset\bR^2$ of prescribed  measure, which allow for
the maximum concentration, is still open. 
In this paper we provide a complete
solution to this problem proving that the optimal sets are balls in phase space, and, in
dimension one, our result can be stated as follows (see Theorem \ref{thm mult} for 
the same result in arbitrary dimension).
\begin{theorem}[Faber-Krahn inequality for the STFT]\label{thm mainthm}
Among all measurable subsets $\Omega\subset \bR^2$ having a prescribed
(finite, non zero) measure,
the quantity
\begin{equation}
\label{eee}
\Phi_\Omega:=
\max_{f\in L^2(\bR)\setminus\{0\}}
\frac{\int_\Omega |\cV f(x,\omega)|^2\, dxd\omega}{\int_{\bR^2} |\cV f(x,\omega)|^2\, dxd\omega}
=
\max_{f\in L^2(\bR)\setminus\{0\}}
\frac{\langle \cV^\ast \mathbbm{1}_\Omega \cV f,f\rangle}{\|f\|^2_{L^2}}
\end{equation}
achieves its maximum if and only if $\Omega$ is equivalent, up to a set of measure zero,
to a ball.

Moreover, when $\Omega$ is a ball of center $(x_0,\omega_0)$, 
the only functions $f$ that achieve the maximum in \eqref{eee} are the functions of the kind
\begin{equation}
\label{optf}
f(x)=c\, e^{2\pi i \omega_0 x }\varphi(x-x_0),\qquad c\in\bC\setminus\{0\},
\end{equation}
that is, the scalar multiples of the Gaussian window $\varphi$ defined in \eqref{defvarphi}, translated and modulated
according to $(x_0,\omega_0)$. 
\end{theorem}
This ``Faber--Krahn inequality'' (see Remark \ref{remFK} at the end of this section)
proves, in the $L^2$-case, a conjecture
by Abreu and Speckbacher \cite{abreu2018} (the full conjecture 
is proved in Theorem \ref{thm lpconc}),
and confirms the distinguished role played by the Gaussian \eqref{optf}, as the first eigenfunction
of the operator $\cV^\ast \mathbbm{1}_\Omega \cV$ when
$\Omega$ has radial symmetry (see \cite{daubechies}; see also \cite{donoho1989} for a related conjecture on band-limited functions,
and \cite[page 162]{cowling} for further insight).

When $\Omega$ is a ball of radius $r$, one can see that $\PhiOm=1-e^{-\pi r^2}$
(this follows from the results in \cite{daubechies}, and will also follow from our  proof of Theorem \ref{thm mainthm}).
Hence we deduce a more explicit form of our result,
which leads to a sharp form of the uncertainty principle for the STFT.

\begin{theorem}[Sharp uncertainty principle for the STFT]\label{cor maincor}
For every subset $\Omega\subset\bR^2$ whose Lebesgue measure $|\Omega|$ is finite we have 
\begin{equation}\label{eq stima 0}
\PhiOm\leq 1-e^{-|\Omega|}
\end{equation}
and, if $|\Omega|>0$, equality occurs if and only if $\Omega$ is a ball. 

As a consequence, if for some $\epsilon\in (0,1)$, some function $f\in L^2(\bR)\setminus\{0\}$ 
 and some $\Omega\subset\bR^2$ we have $\PhiOmega{f}\geq 1-\epsilon$, then necessarily
\begin{equation}\label{eq stima eps}
|\Omega|\geq \log(1/\epsilon),
\end{equation}
with equality   if and only if $\Omega$ is a ball and $f$ has the form \eqref{optf},
 where $(x_0,\omega_0)$ is the center of the ball. 
\end{theorem}

Theorem \ref{cor maincor} solves the long--standing problem of the optimal lower bound for the measure of the ``essential support" of the STFT with Gaussian window.
The best result so far in this direction was obtained by Gr\"ochenig (see
\cite[Theorem 3.3.3]{grochenig-book}) as a consequence of Lieb's uncertainly inequality \cite{lieb} for the STFT, and consists of the following (rougher, but valid for any window) lower bound
\begin{equation}\label{eq statart}
|\Omega|\geq \sup_{p>2}\,(1-\epsilon)^{p/(p-2)}(p/2)^{2/(p-2)}
\end{equation}
(see Section \ref{sec genaralizations} for a discussion in dimension $d$).
Notice that the $\sup$ in \eqref{eq statart} is a bounded function of $\epsilon\in (0,1)$, as 
opposite to the optimal bound in \eqref{eq stima eps} (see Fig.~\ref{figure1} in the
Appendix for a graphical comparison).

We point out that, although in this introduction the discussion of our results
is confined (for ease of notation and exposition) to the one dimensional case, our results are
  valid in arbitrary space dimension, as discussed in Section \ref{sec mult} (Theorem
\ref{thm mult} and Corollary \ref{cor cor2}).

While addressing the reader to \cite{bonami,folland,grochenig} for a review of the numerous uncertainty principles available for the STFT (see also \cite{boggiatto,degosson,demange2005,galbis2010}),
we observe
that  inequality \eqref{eq stima 0} is nontrivial even when $\Omega$ has radial
symmetry: in this particular case it was proved in \cite{galbis2021}, exploiting the already mentioned diagonal representation in the Hermite basis.

Some concentration--type estimates were recently provided in \cite{abreu2018} as an application of the Donoho-Logan large sieve principle \cite{donoho1992} and the Selberg-Bombieri inequality \cite{bombieri}. However, though this machinery certainly has a broad applicability, as observed in \cite{abreu2018} it does not seem to give sharp bounds for the problem above. For interesting applications to signal recovery we refer to \cite{abreu2019,pfander2010,pfander2013,tao} and the references therein. 

Our proof of Theorem \ref{thm mainthm} (and of its multidimensional analogue Theorem \ref{thm mult})
is based on techniques from measure theory, after
the problem has been rephrased as an equivalent statement (where the STFT is no longer involved
explicitly)
in the Fock space. In order to present our strategy in a clear way and to
better highlight the  main ideas, we devote Section \ref{sec proof} to a detailed proof of our main
results in dimension one, while the results in arbitrary dimension
 are stated and proved in Section \ref{sec mult}, focusing on all those things that need to be changed and adjusted.

In Section \ref{sec genaralizations} we discuss some extensions of the above results in different directions,
such as a local version of Lieb's uncertainty inequality for the STFT in $L^p$ when  $p\in [2,\infty)$ (Theorem \ref{thm locallieb}), and $L^p$-concentration estimates for the STFT 
when $p\in [1,\infty)$ (Theorem \ref{thm lpconc}, which proves \cite[Conjecture 1]{abreu2018}), 
identifying in all cases the extremals $f$ and $\Omega$, as above. 
We also study the effect of changing the window $\varphi$ by a dilation or, more generally, by a metaplectic operator.

 We believe that the techniques used in this paper could also shed new light on the Donoho-Stark uncertainty principle \cite{donoho1989} and the corresponding conjecture \cite[Conjecture 1]{donoho1989}, and that also the stability
of \eqref{eq stima 0} (via a quantitative version when the inequality is strict) can be
investigated.
We will address these issues in a subsequent work, together with applications to signal recovery.

\begin{remark}\label{remFK}
The maximization of $\PhiOm$ among all sets $\Omega$ of prescribed measure
can be regarded as a \emph{shape optimization} problem (see \cite{bucur}) and, in this respect, 
Theorem \ref{thm mainthm} shares many analogies
with the celebrated Faber-Krahn inequality (beyond the fact that both problems
have the ball as a solution). The latter states that, among all (quasi) open sets
$\Omega$ of given measure, the ball minimizes the first Dirichlet eigenvalue
\[
\lambda_\Omega:=\min_{u\in H^1_0(\Omega)\setminus\{0\}}
\frac{\int_\Omega |\nabla u(z)|^2\,dz}{\int_\Omega u(z)^2\,dz}.
\]
On the other hand, if $T_\Omega:H^1_0(\Omega)\to H^1_0(\Omega)$ is the linear operator
that associates with every (real-valued) $u\in H^1_0(\Omega)$ the weak solution $T_\Omega u\in H^1_0(\Omega)$
of the problem
$-\Delta (T_\Omega u)=u$ in $\Omega$, integrating by parts we have
\[
\int_\Omega  u^2 \,dz=
-\int_\Omega u \Delta(T_\Omega u)\,dz=\int_\Omega \nabla u\cdot \nabla (T_\Omega u)\,dz=\langle T_\Omega
u,u\rangle_{H^1_0},
\]
so that Faber-Krahn can be rephrased by claiming that
\[
\lambda_\Omega^{-1}:=\max_{u\in H^1_0(\Omega)\setminus\{0\}}
\frac{\int_\Omega u(z)^2\,dz}{\int_\Omega |\nabla u(z)|^2\,dz}
=\max_{u\in H^1_0(\Omega)\setminus\{0\}}
\frac{\langle T_\Omega u,u\rangle_{H^1_0}}{\Vert u\Vert^2_{H^1_0}}
\]
is maximized (among  all open sets of  given measure) by the ball. Hence the statement of
Theorem \ref{thm mainthm} 
can be regarded as a Faber-Krahn inequality for the operator  
$\cV^\ast \mathbbm{1}_\Omega \cV$.
\end{remark}

\section{Rephrasing the problem in the Fock space}\label{sec sec2}
It turns out that the optimization problems discussed in the introduction
can be conveniently rephrased in terms of functions in the Fock space on $\bC$.
We address the reader to \cite[Section 3.4]{grochenig-book} and \cite{zhu} for more details on the relevant results that we are going to review, in a self-contained form, in this section. 

The Bargmann transform of a function $f\in L^2(\bR)$ is defined as 
\[
\cB f(z):= 2^{1/4} \int_\bR f(y) e^{2\pi yz-\pi y^2-\frac{\pi}{2}z^2}\, dy,\qquad z\in\bC.
\]
It turns out that $\cB f(z)$ is an entire holomorphic function and $\cB$ is a unitary operator from $L^2(\bR)$ to the Fock space $\cF^2(\bC)$ of all holomorphic functions $F:\bC\to\bC$ such that 
\begin{equation}\label{defHL}
\|f\|_{\cF^2}:=\Big(\int_\bC |F(z)|^2 e^{-\pi|z|^2}dz\Big)^{1/2}<\infty.
\end{equation}
In fact, $\cB$ maps the orthonormal basis of Hermite functions in $\bR$ into the orthonormal basis of $\cF^2(\bC)$ given by the monomials
\begin{equation}\label{eq ek}
e_k(z):=\Big(\frac{\pi^k}{k!}\Big)^{1/2} z^k,\qquad k=0,1,2,\ldots; \quad z\in\bC. 
\end{equation}
In particular, for the first Hermite function $\varphi(x)=2^{1/4}e^{-\pi x^2}$, that is, the window in \eqref{defvarphi}, we have $\cB \varphi(z)=e_0(z)=1$. 

The connection with the STFT is based on the following crucial formula (see e.g. \cite[Formula (3.30)]{grochenig-book}):
\begin{equation}\label{eq STFTbar}
\cV f(x,-\omega)=e^{\pi i x\omega} \cB f(z) e^{-\pi|z|^2/2},\qquad z=x+i\omega,
\end{equation}
which allows one to rephrase the functionals in \eqref{defphiomegaf} as 
\[
\PhiOmega{f}=\frac{\int_\Omega |\cV f(x,\omega)|^2\, dxd\omega}{\|f\|^2_{L^2}}=
\frac{\int_{\Omega'}|\cB f(z)|^2e^{-\pi|z|^2}\, dz}{\|\cB f\|^2_{\cF^2}}
\]
where $\Omega'=\{(x,\omega):\ (x,-\omega)\in\Omega\}$. Since $\cB:L^2(\bR)\to\cF^2(\bC)$ is a unitary operator, we can safely transfer the optimization problem in Theorem \ref{thm mainthm} directly on 
$\cF^2(\bC)$, observing that
\begin{equation}\label{eq max comp}
\Phi_\Omega=
\max_{F\in\cF^2(\bC)\setminus\{0\}} \frac{\int_{\Omega}|F(z)|^2e^{-\pi|z|^2}\, dz}{\|F\|^2_{\cF^2}}.
\end{equation}
 We will adopt this point of view in Theorem \ref{thm36} below. \par
In the meantime, two remarks are in order. First, we claim that the maximum in \eqref{eq max comp} is invariant under translations of the set $\Omega$. To see this, consider for any $z_0\in\bC$, the operator $U_{z_0}$ defined as
\begin{equation}\label{eq Uz_0}
U_{z_0} F(z)=e^{-\pi|z_0|^2 /2} e^{\pi z\overline{z_0}} F(z-z_0).
\end{equation}
The map $z\mapsto U_z$ turns out to be a projective unitary representation of $\bC$ on $\cF^2(\bC)$, satisfying 
\begin{equation}\label{eq transl}
|F(z-z_0)|^2 e^{-\pi|z-z_0|^2}=|U_{z_0} F(z)|^2 e^{-\pi|z|^2},
\end{equation}
which proves our claim. Invariance under rotations in the plane is also immediate. 

Secondly, we observe that the Bargmann transform intertwines the action of the representation $U_z$ with the so-called ``time-frequency shifts": 
\[
\cB M_{-\omega} T_{x} f= e^{-\pi i x\omega} U_z \cB f, \qquad z=x+i\omega
\]
for every $f\in L^2(\bR)$, where $T_{x}f(y):=f(y-x)$ and $M_{\omega}f(y):=e^{2\pi iy\omega}f(y)$ are the translation and modulation operators. This allows us to write down easily the Bargmann transform of the maximizers appearing in Theorem \ref{thm mainthm}, namely $c U_{z_0} e_0$, $c\in\bC\setminus\{0\}$, $z_0\in\bC$. For future reference, we explicitly set 
\begin{equation}\label{eq Fz0}
F_{z_0}(z):=U_{z_0} e_0(z)=e^{-\frac{\pi}{2}|z_0|^2} e^{\pi z\overline{z_0}}, \quad z,z_0\in\bC.
\end{equation}
The following result shows the distinguished role played by the functions $F_{z_0}$ in connection with extremal problems. A proof can be found in \cite[Theorem 2.7]{zhu}. For the sake of completeness we present a short and elementary proof which generalises in higher dimension.

\begin{proposition}\label{pro1}
Let $F\in\cF^2(\bC)$. Then 
\begin{equation}\label{eq bound}
|F(z)|^2 e^{-\pi|z|^2}\leq \|F\|^2_{\cF^2}\qquad \forall z\in\bC,
\end{equation}
and $|F(z)|^2 e^{-\pi|z|^2}$ vanishes at infinity.
Moreover the equality in \eqref{eq bound} occurs at some point $z_0\in\bC$ if and only if $F=cF_{z_0}$ for some $c\in \bC$.
\end{proposition}
\begin{proof}
By homogeneity we can suppose $\|F\|_{\cF^2}=1$, hence $F=\sum_{k\geq0} c_k e_k$ (cf.\ \eqref{eq ek}), with $\sum_{k\geq 0} |c_k|^2=1$. By the Cauchy-Schwarz inequality we obtain
\[
|F(z)|^2\leq \sum_{k\geq 0} |e_k(z)|^2 =\sum_{k\geq0} \frac{\pi^k}{k!}|z|^{2k}=e^{\pi|z|^2} \quad \forall z\in\bC.
\]
Equality in this estimate occurs at some point $z_0\in\bC$ if and only if $c_k=ce^{-\pi |z_0|^2/2}\overline{e_k(z_0)}$, for some $c\in\bC$, $|c|=1$, which gives 
\[
F(z)=  ce^{-\pi|z_0|^2/2}\sum_{k\geq0} \frac{\pi^k}{k!}(z \overline{z_0})^k=cF_{z_0}(z).
\]
Finally, the fact that $|F(z)|^2 e^{-\pi|z|^2}$ vanishes at infinity is clearly true if $F(z)=z^k$, $k\geq0$, and therefore holds for every $F\in \cF^2(\bC)$ by density, because of \eqref{eq bound}. 
\end{proof}

\section{Proof of the main results in dimension $1$}\label{sec proof}
In this section we prove
Theorems \ref{thm mainthm} and \ref{cor maincor}. In fact, 
by the discussion in Section \ref{sec sec2}, cf.\ \eqref{eq max comp}, 
these will follow (without further reference) from the following result,
which will be proved at the end of this section, after a few preliminary results
have been established.

\begin{theorem}\label{thm36}
For every $F\in \cF^2(\bC)\setminus\{0\}$ and every measurable set $\Omega\subset\bR^2$
of finite measure, 
we have
\begin{equation}
\label{stimaquoz}
\frac{\int_\Omega|F(z)|^2 e^{-\pi|z|^2}\, dz}{\|F\|_{\cF^2}^2}
\leq 1-e^{-|\Omega|}.
\end{equation}
Moreover, recalling \eqref{eq Fz0}, equality occurs (for some $F$ and for some $\Omega$ such that
$0<|\Omega|<\infty$) if and only if $F=c F_{z_0}$ (for some
$z_0\in\bC$ and some nonzero $c\in\bC$) and $\Omega$ is equivalent, 
up to a set of measure zero, to
a ball centered at $z_0$.
\end{theorem}

Throughout the rest of this section, in view of proving
\eqref{stimaquoz}, given an arbitrary function
$F\in \cF^2(\bC)\setminus\{0\}$ we shall investigate several
properties of the function
\begin{equation}
\label{defu}
u(z):=|F(z)|^2 e^{-\pi|z|^2},
\end{equation}
in connection with its super-level sets

\begin{equation}
\label{defAt}
A_t:=\{u>t\}=\left\{z\in\bR^2\,:\,\, u(z)>t\right\},
\end{equation}
its \emph{distribution function}

\begin{equation}
\label{defmu}
\mu(t):= |A_t|,\qquad 0\leq t\leq \max_{\bC} u
\end{equation}
(note that $u$ is bounded due to \eqref{eq bound}),
and the \emph{decreasing rearrangement} of $u$, i.e. the function
\begin{equation}
\label{defclassu*}
u^*(s):=\sup\{t\geq 0\,:\,\, \mu(t)>s\}\qquad \text{for $s\geq 0$}
\end{equation}
(for more details on rearrangements, we refer to  \cite{baernstein}).
Since $F(z)$ in \eqref{defu} is entire holomorphic, $u$ (which letting $z=x+i\omega$
can be regarded as a real-valued function
$u(x,\omega)$ on $\bR^2$) has several nice properties which will simplify our analysis. In particular,
$u$ is \emph{real analytic} and hence, since $u$ is not a constant,  \emph{every} level set
of $u$ has zero measure (see e.g. \cite{krantz}), i.e.
\begin{equation}
\label{lszm}
\left| \{u=t\}\right| =0\quad\forall t\geq 0
\end{equation}
and, similarly, the set of all critical points of $u$ has zero measure, i.e.
\begin{equation}
\label{cszm}
\left| \{|\nabla u|=0\}\right| =0.
\end{equation}
Moreover, since by Proposition \ref{pro1} $u(z)\to 0$ as $|z|\to\infty$, by Sard's
Lemma we see that for a.e. $t\in (0,\max u)$ the super-level set $\{u>t\}$ is a bounded 
open set in $\bR^2$ with smooth boundary
\begin{equation}
\label{boundaryAt}
\partial\{u>t\}=\{u=t\}\quad\text{for a.e. $t\in (0,\max u).$}
\end{equation} 
Since $u(z)>0$ a.e. (in fact everywhere, except at most at isolated points), 
\[
\mu(0)=\lim_{t\to 0^+}\mu(t)=+\infty,
\]
while
the finiteness of $\mu(t)$ when $t\in (0,\max u]$ is entailed by the fact that $u\in L^1(\bR^2)$, 
according to \eqref{defu} and \eqref{defHL} (in particular $\mu(\max u)=0$).
Moreover, by \eqref{lszm}  $\mu(t)$ is \emph{continuous} (and not
just right-continuous) at \emph{every point} $t\in (0,\max u]$.
Since $\mu$ is also strictly decreasing, we see that $u^*$, according to \eqref{defclassu*}, 
is just the elementarly defined \emph{inverse function} of $\mu$ (restricted to $(0,\max u]$), i.e. 
\begin{equation}
\label{defu*}
u^*(s)=\mu^{-1}(s) \qquad\text{for $s\geq 0$,}
\end{equation}
which maps $[0,+\infty)$ decreasingly and continuously onto $(0,\max u]$.

In the following we will strongly rely on the following result.
\begin{lemma}\label{lemmau*}
The function $\mu$  is absolutely continuous
on the compact subintervals of $(0,\max u]$, and
\begin{equation}
\label{dermu}
-\mu'(t)= \int_{\{u=t\}} |\nabla u|^{-1} \dH \qquad\text{for a.e. $t\in (0,\max u)$.}
\end{equation}
Similarly, 
the function $u^*$ is absolutely continuous on  the compact subintervals of $[0,+\infty)$, and
\begin{equation}
\label{deru*} -(u^*)'(s)=
\left(\int_{\{u=u^*(s)\}} |\nabla u|^{-1} \dH\right)^{-1} \qquad\text{for a.e. $s\geq 0$.}
\end{equation}
\end{lemma}
These properties of $\mu$ and $u^*$ are essentially well known to the specialists in
rearrangement theory, and follow e.g. from the general results
of \cite{almgren-lieb,BZ}, which are valid within the framework of $W^{1,p}$ functions
(see also \cite{cianchi} for the framework of $BV$ functions, in particular Lemmas 3.1 and
3.2). We point out, however, that 
of these properties only the absolute continuity of $u^*$ is valid in general, 
 while the others
strongly depend on \eqref{cszm} which, in the terminology of \cite{almgren-lieb},
 implies that $u$ is \emph{coarea regular} in a very strong sense, since  it
 rules out the possibility of a singular part in the (negative) Radon measure $\mu'(t)$ and, 
 at the same time, it
 guarantees that the density of the absolutely continuous part is given (only) by the right-hand side of
 \eqref{dermu}. As clearly explained in the excellent Introduction to \cite{almgren-lieb},
 there are several subtleties related to the structure of the distributional derivative of
 $\mu(t)$ (which ultimately make the validity of \eqref{deru*} highly nontrivial), and in fact
 the seminal paper \cite{BZ}  was motivated by a subtle error in a previous work, whose
 fixing since \cite{BZ} has stimulated a lot of original and deep research
 (see e.g. \cite{cianchi,fuscoAnnals} and references therein). 
  
 However, since unfortunately we were not able to find a ready-to-use reference for \eqref{deru*}
 (and, moreover, our $u$ is very smooth but strictly speaking
 it does not belong to $W^{1,1}(\bR^2)$, which would require to fix a lot of details
 when referring to the general results from \cite{almgren-lieb,BZ,cianchi}), here we present
 an elementary and self-contained proof of this lemma, specializing to our case a general
 argument from \cite{BZ} based on the coarea formula.

\begin{proof}[Proof of Lemma \ref{lemmau*}]
The fact that $u$ is locally Lipschitz guarantees the
validity of  the coarea formula (see e.g. \cite{BZ,evans}), that is,
for 
every Borel function $h:\bR^2\to [0,+\infty]$ we have
\[
\int_{\bR^2} h(z) |\nabla u(z)|\,dz = 
\int_0^{\max u} \left( \int_{\{u=\tau\}} h \dH\right)\,d\tau,
\]
where ${\mathcal H}^1$ denotes the one-dimensional Hausdorff measure (and with the usual
convention that $0\cdot \infty=0$ in the first integral). In particular, when
$h(z)=\chi_{A_t}(z) |\nabla u(z)|^{-1}$ (where $|\nabla u(z)|^{-1}$ is meant as $+\infty$
 if $z$ is a critical point of $u$), by virtue of \eqref{cszm} the function $h(z)|\nabla u(z)|$
coincides with $\chi_{A_t}(z)$ a.e., and recalling \eqref{defmu} 
 one obtains
\begin{equation}
\label{rappmu}
\mu(t)=\int_t^{\max u} \left( \int_{\{u=\tau\}} |\nabla u|^{-1} \dH 
\right)\,d\tau\qquad\forall t\in [0,\max u];
\end{equation}
therefore we see that $\mu(t)$ is \emph{absolutely continuous}
on the compact subintervals of $(0,\max u]$, and 
\eqref{dermu} follows.

Now let $D\subseteq (0,\max u)$ denote the set where $\mu'(t)$ exists, coincides with the
integral in \eqref{dermu}
and is strictly positive, and let $D_0=(0,\max u]\setminus D$. By \eqref{dermu} and 
the absolute continuity of $\mu$,
and since the integral in \eqref{dermu} is strictly positive for \emph{every} $t\in (0,\max u)$
(note that  ${\mathcal H}^1(\{u=t\})>0$ for every $t\in (0,\max u)$, otherwise we would have
that $|\{u>t\}|=0$ by the isoperimetric inequality),
we infer that $|D_0|=0$, so that letting
$\widehat D=\mu(D)$ and $\widehat D_0=\mu(D_0)$, one has $|\widehat D_0|=0$ by the absolute
continuity of $\mu$, and $\widehat D=[0,+\infty)\setminus \widehat D_0$ since $\mu$ is invertible. On the other hand, by \eqref{defu*} and elementary calculus,
we see that $(u^*)'(s)$ exists for \emph{every} $s\in \widehat{D}$ and
\[
-(u^*)'(s)=\frac{-1}{\mu'(\mu^{-1}(s))} = 
\left(\int_{\{u=u^*(s)\}} |\nabla u|^{-1} \dH\right)^{-1} \qquad\forall s\in\widehat D,
\]
which implies \eqref{deru*} since $|\widehat D_0|=0$. Finally, since $u^*$ is
differentiable \emph{everywhere} on $\widehat D$, it is well known that
$u^*$ maps every negligible set
$N\subset \widehat D$ into a negligible set. Since $\widehat D\cup \widehat D_0=[0,+\infty)$,
and moreover $u^*(\widehat D_0)=D_0$
where $|D_0|=0$, we see that $u^*$ maps negligible sets into negligible sets, hence it is
absolutely continuous on every compact interval $[0,a]$.
\end{proof} 
The following estimate for the integral in \eqref{deru*},
which can be of some interest in itself, will be the main ingredient in the proof of
Theorem \ref{thm36}.
\begin{proposition}\label{prop34}
We have
\begin{equation}
\label{eq4}
\left(\int_{\{u=u^*(s)\}} |\nabla u|^{-1} \dH\right)^{-1}
\leq u^*(s)\qquad\text{for a.e. $s>0$,}
\end{equation}
and hence
\begin{equation}
\label{stimaderu*}
(u^*)'(s)+ u^*(s)\geq 0\quad\text{for a.e. $s\geq 0$.}
\end{equation}
\end{proposition}
\begin{proof}
Letting for simplicity $t=u^*(s)$ and recalling that, 
for a.e. $t\in (0,\max u)$
(or, equivalently, for a.e. $s>0$, since $u^*$ and its inverse $\mu$
are absolutely continuous on compact sets) the super-level set $A_t$ in \eqref{defAt}
 has a smooth
boundary  as in \eqref{boundaryAt},
we can combine the Cauchy-Schwarz inequality
\begin{equation}
\label{CS}
 {\mathcal H}^1(\{u=t\})^2 \leq
 \left(\int_{\{u=t\}} |\nabla u|^{-1} \dH\right)
\int_{\{u=t\}} |\nabla u| \dH
\end{equation}
with the isoperimetric inequality in the plane
\begin{equation}
\label{isop}
4\pi \,|\{ u > t \}|\leq 
 {\mathcal H}^1(\{u=t\})^2 
 \end{equation}
 to obtain, after division by $t$,
\begin{equation}
\label{eq3}
t^{-1}
\left(\int_{\{u=t\}} |\nabla u|^{-1} \dH\right)^{-1}
\leq 
\frac{\int_{\{u=t\}} \frac{|\nabla u|}t \dH
}{4\pi \,|\{ u > t \}|}.
\end{equation}
The reason for dividing by $t$ is that, in this form, the right-hand side turns out to be
(quite surprisingly, at least to us) independent of $t$. Indeed, 
since along $\partial A_t=\{u=t\}$ we have $|\nabla u|=-\nabla u\cdot \nu$
where $\nu$ is the outer normal to $\partial A_t$, along $\{u=t\}$ we can interpret
the quotient $|\nabla u|/t$ as $-(\nabla\log u)\cdot\nu$, and hence
\begin{equation*}
\int_{\{u=t\}} \frac{|\nabla u|}t \dH
=-\int_{\partial A_t} (\nabla\log u)\cdot\nu \dH
=-\int_{A_t} \Delta \log u(z)\,dz.
\end{equation*}
But by \eqref{defu}, since $\log |F(z)|$ is a harmonic function, we obtain
\begin{equation}
\label{laplog}
\Delta(\log u(z))=
\Delta(\log |F(z)|^2 +\log e^{-\pi |z|^2})
=\Delta (-\pi |z|^2)=-4\pi,
\end{equation}
so that the last integral equals $4\pi |A_t|$. Plugging this into
\eqref{eq3}, one obtains that the quotient on the right equals $1$, and \eqref{eq4}
follows. Finally, \eqref{stimaderu*} follows on combining \eqref{deru*} with \eqref{eq4}.
\end{proof}

The following lemma establishes a link between the integrals of $u$ on its super-level sets
(which will play a major role in our main argument)
and the function $u^*$.
\begin{lemma}\label{lemma3.3}
The function
\begin{equation}
\label{defI}
I(s)=\int_{\{u > u^*(s)\}} u(z)dz,\qquad s\in [0,+\infty),
\end{equation}
i.e. the integral of $u$ on its (unique) super-level set of measure $s$, 
is of class $C^1$ on $[0,+\infty)$, and
\begin{equation}
\label{derI}
I'(s)=u^*(s)\quad\forall s\geq 0.
\end{equation}
Moreover, $I'$ is (locally) absolutely continuous, and
\begin{equation}
\label{derI2}
I''(s)+I'(s)\geq 0\quad \text{for a.e. $s\geq 0$.}
\end{equation}
\end{lemma}
\begin{proof}
We have for every $h>0$ and every $s\geq 0$
\[
I(s+h)-I(s)=
\int_{ \{u^*(s+h)< u\leq u^*(s)\}} u(z)dz
\]
and, since by \eqref{defu*} and \eqref{defmu} $|A_{u^*(\sigma)}|=\sigma$,
\[
\left| \{u^*(s+h)< u\leq u^*(s)\}\right| =
|A_{u^*(s+h)}|-|A_{u^*(s)}|=(s+h)-s=h,
\]
we obtain
\[
u^*(s+h) \leq \frac{I(s+h)-I(s)}{h}\leq u^*(s).
\]
Moreover, it is easy to see that the same inequality is true also when $h<0$ (provided $s+h>0$),
now using the reverse set inclusion $A_{u^*(s+h)}\subset A_{u^*(s)}$ according to
the fact that $u^*$ is decreasing. Since $u^*$ is continuous, \eqref{derI} follows letting $h\to 0$
when $s>0$, and letting $h\to 0^+$ when $s=0$.

Finally, by Lemma \ref{lemmau*}, $I'=u^*$ is absolutely continuous on $[0,a]$ for every $a\geq 0$,
$I''=(u^*)'$, and \eqref{derI2} follows from \eqref{stimaderu*}.
\end{proof}

We are now in a position to prove Theorem \ref{thm36}.

\begin{proof}[Proof of Theorem \ref{thm36}]
By homogeneity we can assume $\|F\|_{\cF^2}=1$ so that,
defining $u$ as in \eqref{defu}, \eqref{stimaquoz} is equivalent to 
\begin{equation}
\label{eq1}
\int_\Omega u(z)\,dz \leq 1-e^{-s}
\end{equation}
for every $s\geq 0$ and every $\Omega\subset\bR^2$ such that $|\Omega|=s$.
It is clear that, for any fixed measure $s\geq 0$,
the integral on the left is maximized
when $\Omega$ is the (unique by \eqref{lszm}) super-level set $A_t=\{u>t\}$ 
such that $|A_t|=s$ (i.e. $\mu(t)=s$), and by \eqref{defu*} we see 
that the proper cut level is given by $t=u^*(s)$.
In other words, if $|\Omega|=s$ then
\begin{equation}
\label{eq2}
\int_\Omega u(z)\,dz\leq \int_{A_{u^*(s)}} u(z)\,dz,
\end{equation}
with strict inequality unless $\Omega$ coincides --up to a negligible set-- with $A_{u^*(s)}$
(to see this, it suffices to let $E:=\Omega\cap A_{u^*(s)}$ and observe that, 
if $|\Omega\setminus E|> 0$, then the integral of
$u$ on $\Omega\setminus E$, where $u\leq u^*(s)$, is strictly smaller than the integral of $u$
on $A_{u^*(s)}\setminus E$, where $u> u^*(s)$).
Thus,
to prove \eqref{stimaquoz}
it suffices to prove  \eqref{eq1} when $\Omega=A_{u^*(s)}$, that is,
recalling \eqref{defI}, prove that 
\begin{equation}
\label{ineqI}
I(s)\leq 1-e^{-s}\qquad\forall s\geq 0
\end{equation}
or, equivalently, letting $s=-\log \sigma$, that
\begin{equation}
\label{ineqI2}
G(\sigma):= I(-\log \sigma)\leq 1-\sigma \qquad\forall \sigma\in (0,1].
\end{equation}
Note that
\begin{equation}
\label{v0}
G(1)=I(0)=\int_{\{u>u^*(0)\}} u(z)\,dz = \int_{\{u>\max u\}} u(z)\,dz=0,
\end{equation}
while by monotone convergence, since $\lim_{s\to+\infty} u^*(s)=0$,
\begin{equation}
\label{vinf}
\lim_{\sigma\to 0^+} G(\sigma)=
\lim_{s\to+\infty} I(s)=
\int_{\{u>0\}}\!\!\! u(z)\,dz 
=
\int_{\bR^2} |F(z)|^2 e^{-\pi |z|^2}\,dz=1,
\end{equation}
because we assumed $F$ is normalized. Thus, $G$ extends to a continuous function on $[0,1]$
that coincides with $1-\sigma$ at the endpoints, and \eqref{ineqI2} will follow by proving that $G$
is convex. Indeed, by \eqref{derI2}, the function $e^s I'(s)$ is non decreasing, and since 
$G'(e^{-s})=-e^s I'(s)$, this means that $G'(\sigma)$ is non decreasing as well, i.e. $G$ is convex as claimed.

Summing up, via \eqref{eq2} and \eqref{ineqI}, we have proved that
for every $s\geq 0$
\begin{equation}
\label{sumup}
\begin{split}
&\int_\Omega|F(z)|^2 e^{-\pi|z|^2}\, dz
=\int_\Omega u(z)\,dz \\
\leq &\int_{A_{u^*(s)}} u(z)\,dz=I(s)\leq 1-e^{-s}
\end{split}
\end{equation}
for every $F$ such that $\|F\|_{\cF^2}=1$.

 Now assume 
that equality occurs in \eqref{stimaquoz}, for some $F$ (we may still assume
$\|F\|_{\cF^2}=1$) and
for some set $\Omega$ of measure
$s_0>0$: then, when $s=s_0$, equality occurs everywhere in \eqref{sumup},
i.e. in \eqref{eq2}, whence $\Omega$ coincides with $A_{u^*(s_0)}$ up to a set of measure
zero, and in \eqref{ineqI}, whence $I(s_0)=1-e^{-s_0}$. But then $G(\sigma_0)=1-\sigma_0$ in
\eqref{ineqI2}, where $\sigma_0=e^{-s_0}\in (0,1)$: since $G$ is convex on $[0,1]$, and coincides with
$1-\sigma$ at the endpoints, we infer that $G(\sigma)=1-\sigma$ for every $\sigma\in [0,1]$, or, equivalently,
that $I(s)=1-e^{-s}$ for \emph{every} $s\geq 0$. In particular, $I'(0)=1$; on the other hand,
choosing $s=0$ in \eqref{derI} gives
\[
I'(0)=u^*(0)=\max u,
\]
so that $\max u=1$. But then by \eqref{eq bound}
\begin{equation}
\label{catena}
1=\max u =\max |F(z)|^2 e^{-\pi |z|^2}\leq \|F\|^2_{\cF^2}=1
\end{equation}
and, since equality is attained, by Proposition \ref{pro1} we infer that $F=c F_{z_0}$ for
some $z_0,c\in\bC$. We have already proved that $\Omega=A_{u^*(s_0)}$ (up to a negligible set)
and, since by \eqref{eq Fz0}
\begin{equation}
\label{uradial}
u(z)=|c F_{z_0}(z)|^2 e^{-\pi |z|^2}
=|c|^2 e^{-\pi |z_0|^2} e^{2\pi\realp  (z \overline{z_0})}e^{-\pi |z|^2}=|c|^2 e^{-\pi |z-z_0|^2}
\end{equation}
has radial symmetry about $z_0$ and is radially decreasing, $\Omega$ is (equivalent to) a
ball centered at $z_0$.
This proves the ``only if part" of the final claim being proved. 

The ``if part'' follows by a direct computation.
For, 
assume that $F=c F_{z_0}$ and $\Omega$ is equivalent to a ball of radius $r>0$ centered
at $z_0$. Then
using \eqref{uradial} 
we
can compute, using polar coordinates
\[
\int_\Omega u(z)\,dz=
|c|^2 \int_{\{|z|<r\}} e^{-\pi |z|^2}\,dz =
2\pi |c|^2\int_0^\rho \rho e^{-\pi \rho^2}\,d\rho=|c|^2(1-e^{-\pi r^2}),
\]
and equality occurs in \eqref{stimaquoz} because $\|c F_{z_0}\|_{\cF^2}^2=|c|^2$.
\end{proof}

\begin{remark}
The ``only if part" in the final claim of Theorem \ref{thm36}, once one has established
that $I(s)=1-e^{-s}$ for every $s\geq 0$, instead of using \eqref{catena}, can also be
proved observing that there must be equality, for a.e. $t\in (0,\max u)$, both in
\eqref{CS} and in \eqref{isop} (otherwise there would be a strict inequality in \eqref{stimaderu*},
hence also in \eqref{ineqI},
on a set of positive measure). But then, for at least one value  (in fact, for infinitely many
values) of $t$ we would have that $A_t$ is a ball $B(z_0,r)$ (by the equality in the isoperimetric
estimate \eqref{isop}) and that $|\nabla u|$ is constant along $\partial A_t=\{u=t\}$
(by the equality in \eqref{CS}).

By applying the ``translation'' $U_{z_0}$  (cf.\ \eqref{eq Uz_0} and \eqref{eq transl}) 
we can suppose that the super-level set $A_t=B(z_0,r)$ is centred at the origin, i.e. that $z_0=0$, and
in that case we have to prove that $F$ is constant (so that, translating back to $z_0$, 
one obtains that the original $F$ had the form $c F_{z_0}$). 
Since now both $u$ and $e^{-|z|^2}$ are constant along $\partial A_t=\partial B(0,r)$, 
also $|F|$ is constant there (and does not vanish inside $\overline{B(0,r)}$,
since $u\geq t>0$ there).  Hence $\log|F|$ is constant along $\partial B(0,r)$,
and is harmonic inside $B(0,r)$ since $F$ is holomorphic:
therefore $\log |F|$ is constant in $B(0,r)$, which implies that $F$ is constant over $\bC$. 

Note that the constancy of $|\nabla u|$ along $\partial A_t$ has not been used. However, also this property alone
(even ignoring that $A_t$ is a ball) is enough to conclude. Letting $w=\log u$, one can use that both
$w$ and $|\nabla w|$ are constant
along $\partial A_t$, and moreover $\Delta w=-4\pi$ as shown in \eqref{laplog}: hence every connected
component of $A_t$ must be a ball, by a celebrated result of Serrin \cite{serrin}. Then the previous
argument can be applied to just one connected component of $A_t$, which is a ball,
to conclude that $F$ is constant.
\end{remark}

\section{The multidimensional case}\label{sec mult}
In this Section we provide the generalisation of Theorems \ref{thm mainthm} and
\ref{cor maincor} (in fact, of Theorem \ref{thm36}) in arbitrary dimension. 

We recall that 
the STFT of a function $f\in L^2(\bR^d)$, with a given window $g\in L^2(\bR^d)\setminus\{0\}$, is defined as 
\begin{equation}\label{eq STFT wind}
\cV_g f(x,\omega):=\int_{\bR^d} e^{-2\pi i y\cdot\omega} f(y)\overline{g(y-x)}\, dy,\qquad x,\omega\in\bR^d.
\end{equation}
Consider now the Gaussian function 
\begin{equation}\label{eq gaussian dimd}
\varphi(x)=2^{-d/4}e^{-\pi|x|^2}\qquad x\in\bR^d,
\end{equation} and the corresponding STFT in \eqref{eq STFT wind} with window $g=\varphi$; let us write shortly $\cV=\cV_\varphi$.
Let $\boldsymbol{\omega}_{2d}$ be the measure of the unit ball in $\bR^{2d}$. Recall also the definition of the (lower) incomplete $\gamma$ function as 
\begin{equation}
\label{defgamma}
\gamma(k,s):=\int_0^s \tau^{k-1}e^{-\tau}\, d\tau
\end{equation}
where $k\geq 1$ is an integer and $s\geq 0$, so that
\begin{equation}
\label{propgamma}
\frac{\gamma(k,s)}{(k-1)!}= 1-e^{-s}\sum_{j=0}^{k-1} \frac{s^j}{j!}.
\end{equation}

\begin{theorem}[Faber--Krahn inequality for the STFT in dimension $d$]\label{thm mult} 
For every measurable subset $\Omega\subset\bR^{2d}$ of finite measure and for
every $f\in L^2(\bR^d)\setminus\{0\}$ there holds
\begin{equation}\label{eq thm mult}
\frac{\int_\Omega |\cV f(x,\omega)|^2\, dxd\omega}{\|f\|^2_{L^2}}\leq \frac{\gamma(d,c_\Omega)}{(d-1)!},
\end{equation}
where $c_\Omega:=\pi(|\Omega|/\boldsymbol{\omega}_{2d})^{1/d}$ is
 the symplectic capacity of the ball in $\bR^{2d}$ having the same volume as $\Omega$. 

Moreover,  equality occurs (for some $f$ and for some $\Omega$ such that
$0<|\Omega|<\infty$) if and only if 
$\Omega$ is equivalent, 
up to a set of measure zero, to
a ball centered at some $(x_0,\omega_0)\in\bR^{2d}$, and
\begin{equation}\label{optf-bis}
f(x)=ce^{2\pi ix\cdot\omega_0}\varphi(x-x_0),\qquad c\in\bC\setminus\{0\},
\end{equation}
where $\varphi$ is the Gaussian in \eqref{eq gaussian dimd}.
 \end{theorem}
We recall that the symplectic capacity of a ball of radius $r$ in phase space is $\pi r^2$ in every dimension and represents the natural measure of the size of the ball from the point of view of the symplectic geometry \cite{degosson,gromov,hofer}. 
\begin{proof}[Proof of Theorem \ref{thm mult}]
We give only a sketch of the proof, because it follows the same pattern as in dimension $1$. \par
The definition of the Fock space $\cF^2(\bC)$ extends essentially verbatim to $\bC^d$, with the monomials $(\pi^{|\alpha|}/\alpha!)^{1/2}z^\alpha$, $z\in\bC^d$, $\alpha\in\bN^d$ (multi-index notation) as orthonormal basis. The same holds for the definition of the functions $F_{z_0}$ in \eqref{eq Fz0}, now with $z,z_0\in\bC^d$, and Proposition \ref{pro1} extends in the obvious way too. Again one can rewrite the optimization problem in the Fock space $\cF^2(\bC^d)$, the formula \eqref{eq STFTbar} continuing to hold, with $x,\omega\in\bR^d$. Hence we have to prove that 
\begin{equation}
\label{stimaquoz bis}
\frac{\int_\Omega|F(z)|^2 e^{-\pi|z|^2}\, dz}{\|F\|_{\cF^2}^2}
\leq \frac{\gamma(d,c_\Omega)}{(d-1)!}
\end{equation}
for $F\in \cF^2(\bC^d)\setminus\{0\}$ and $\Omega\subset\bC^{d}$ of finite measure, and that equality occurs if and only if $F=c F_{z_0}$ and $\Omega$ is equivalent, up to a set of measure zero, to a ball centered at $z_0$.

To this end, for $F\in \cF^2(\bC^d)\setminus\{0\}$, $\|F\|_{\cF^2}=1$, we set $u(z)=|F(z)|^2 e^{-\pi|z|^2}$, $z\in\bC^d$,
exactly as in \eqref{defu} when $d=1$, and define $A_t$, $\mu(t)$ and $u^*(s)$ as in
Section \ref{sec proof}, replacing $\bR^{2}$ with $\bR^{2d}$ where necessary,
now denoting by $|E|$ the $2d$-dimensional Lebesgue measure  of 
a set $E\subset\bR^{2d}$, in place of the
2-dimensional measure. Note that, now regarding $u$ as a function of $2d$ real variables in $\bR^{2d}$,
properties \eqref{lszm}, \eqref{cszm} etc. are still valid, 
as well as  formulas \eqref{dermu}, \eqref{deru*} etc., provided one
 replaces every occurrence of $\cH^1$ with the $(2d-1)$-dimensional Hausdorff measure $\cH^{2d-1}$.  Following the same pattern as in Proposition \ref{prop34}, now using the isoperimetric inequality in $\bR^{2d}$ (see e.g. \cite{fusco-iso} for an updated account)
\[
\cH^{2d-1}(\{u=t\})^2\geq (2d)^2\boldsymbol{\omega}_{2d}^{1/d}|\{u>t\}|^{(2d-1)/d}
\]
and the fact that $\triangle \log u=-4\pi d$ on $\{u>0\}$, we see that now
$u^\ast$ satisfies the inequality
\[
\left(\int_{\{u=u^*(s)\}} |\nabla u|^{-1} \, d\cH^{2d-1}\right)^{-1}
\leq \pi d^{-1}\boldsymbol{\omega}_{2d}^{-1/d} s^{-1+1/d} u^*(s)\quad\text{for a.e. $s>0$}
\]
in place of \eqref{eq4},
and hence \eqref{stimaderu*} is to be replaced with
\[
(u^*)'(s)+ \pi d^{-1}\boldsymbol{\omega}_{2d}^{-1/d} s^{-1+1/d} u^*(s)\geq 0\quad\text{for a.e. $s> 0$.}
\]
Therefore, with the notation of Lemma \ref{lemma3.3}, $I'(t)$ is locally absolutely continuous on $[0,+\infty)$ and now satisfies
\[
I''(s)+ \pi d^{-1}\boldsymbol{\omega}_{2d}^{-1/d} s^{-1+1/d} I'(s)\geq 0\quad\text{for a.e. $s> 0$.}
\]
This implies that the function $e^{\pi \boldsymbol{\omega}_{2d}^{-1/d} s^{1/d}}I'(s)$ is non decreasing on $[0,+\infty)$.
Then, arguing as in the proof of Theorem \ref{thm36}, we are led to prove  the inequality 
\[
I(s)\leq \frac{\gamma(d,\pi (s/\boldsymbol{\omega}_{2d})^{1/d})}{(d-1)!},\qquad s\geq0
\]
in place of \eqref{ineqI}.
This, with the substitution 
\[
\gamma(d,\pi (s/\boldsymbol{\omega}_{2d})^{1/d})/(d-1)!=1-\sigma,\qquad \sigma\in (0,1]
\]
(recall \eqref{propgamma}), turns into
\[
G(\sigma):=I(s)\leq 1-\sigma\quad \forall\sigma\in(0,1]. 
\]
Again $G$ extends to a continuous function on $[0,1]$, with $G(0)=1$, $G(1)=0$.
At this point one observes that, regarding $\sigma$ as a function of $s$,
\[
G'(\sigma(s))=-d! \pi^{-d}\boldsymbol{\omega}_{2d}  e^{\pi (s/\boldsymbol{\omega}_{2d})^{1/d}}I'(s).
\]
Since the function $e^{\pi (s/\boldsymbol{\omega}_{2d})^{1/d}}I'(s)$ is non decreasing, we see that $G'$ is non increasing on $(0,1]$, hence $G$ is convex on $[0,1]$ and one concludes as in the proof of Theorem \ref{thm36}. Finally, the ``if part" follows from a direct computation, similar to that
at the end of the proof of Theorem \ref{thm36}, now integrating on a ball in dimension $2d$,
and using \eqref{defgamma} to evaluate the resulting integral.
\end{proof}
As a consequence of Theorem \ref{thm mult} we deduce a sharp form of the uncertainty principle for the STFT, which generalises Theorem \ref{cor maincor} to arbitrary dimension. To replace the function
$\log(1/\epsilon)$ in \eqref{eq stima eps} (arising as the inverse function of $e^{-s}$
in the right-hand side 
of \eqref{eq stima 0}), we now denote
by 
$\psi_d(\epsilon)$,  $0<\epsilon\leq1$, the inverse function of 
\[
s\mapsto 1-\frac{\gamma(d,s)}{(d-1)!}=e^{-s}\sum_{j=0}^{d-1} \frac{s^j}{j!},\qquad s\geq 0
\]
(cf. \eqref{propgamma}).
\begin{corollary}\label{cor cor2} If for some $\epsilon\in (0,1)$, some 
 $f\in L^2(\bR^d)\setminus\{0\}$, and some $\Omega\subset\bR^{2d}$ we have
  $\int_\Omega |\cV f(x,\omega)|^2\, dxd\omega\geq (1-\epsilon) \|f\|^2_{L^2}$, then 
\begin{equation}\label{uncertainty dim d}
|\Omega|\geq \boldsymbol{\omega}_{2d}\pi^{-d}\psi_d(\epsilon)^d,
\end{equation}
with equality if and only if $\Omega$ is a ball and $f$ has the form \eqref{optf-bis},
 where $(x_0,\omega_0)$ is the center of the ball.
\end{corollary}
So far, the state-of-the-art in this connection has been represented by the lower bound 
\begin{equation}\label{bound groc dim d}
|\Omega|\geq \sup_{p>2}\,(1-\epsilon)^{p/(p-2)}(p/2)^{2d/(p-2)}
\end{equation}
(which reduces to \eqref{eq statart} when $d=1$, see \cite[Theorem 3.3.3]{grochenig-book}). See Figure \ref{figure1} in the Appendix for a graphical comparison with \eqref{uncertainty dim d} in dimension $d=2$. 
 Figure \ref{figure2} in the Appendix illustrates Theorem \ref{thm mult} and Corollary \ref{cor cor2}.

\begin{remark*} Notice that $\psi_1(\epsilon)=\log(1/\epsilon)$, and $\psi_d(\epsilon)$ is 
increasing with $d$. Moreover, it is easy to check that
\begin{align*}
\psi_d(\epsilon)&\sim (d!)^{1/d}(1-\epsilon)^{1/d},\quad \epsilon\to 1^-\\
\psi_d(\epsilon)&\sim \log(1/\epsilon),\quad \epsilon \to 0^+. 
\end{align*}
On the contrary, the right-hand side of \eqref{bound groc dim d} is bounded by $e^d$; see Figure \ref{figure1} in the Appendix.
\end{remark*}

\section{Some generalizations}\label{sec genaralizations}
In this Section we discuss some generalizations in several directions. 
\subsection{Local Lieb's uncertainty inequality for the STFT} 
An interesting variation on the theme is given by the optimization problem
\begin{equation}\label{eq phip}
\sup_{f\in {L^2(\bR)\setminus\{0\}}}\frac{\int_\Omega |\cV f(x,\omega)|^p\, dxd\omega}{\|f\|^p_{L^2}},
\end{equation}
where $\Omega\subset\bR^2$ is measurable subset of finite measure and $2\leq p<\infty$. Again, we look for the subsets $\Omega$, of prescribed measure, which maximize the above supremum.

Observe, first of all, that by the Cauchy-Schwarz inequality, $\|\cV f\|_{L^\infty}\leq \|f\|_{L^2}$, so that the supremum in \eqref{eq phip} is finite and, in fact, it is attained.
\begin{proposition}\label{pro41}
The supremum in \eqref{eq phip} is attained.
\end{proposition}
\begin{proof} The desired conclusion follows easily by the direct method of the calculus of variations. We first rewrite the problem in the complex domain via \eqref{eq STFTbar}, as we did in Section \ref{sec sec2}, now ending up with the Rayleigh quotient
\[
\frac{\int_\Omega |F(z)|^p e^{-p\pi|z|^2/2}\, dz}{\|F\|^p_{\cF^2}}
\]
with $F\in \cF^2(\bC)\setminus\{0\}$. It is easy to see that this expression attains a maximum at some $F\in\cF^2(\bC)\setminus\{0\}$. In fact, let $F_n\in \cF^2(\bC)$, $\|F_n\|_{\cF^2}=1$, be a maximizing sequence, and let $u_n(z)= |F_n(z)|^p e^{-p\pi|z|^2/2}$. Since $u_n(z)= (|F_n(z)|^2 e^{-\pi|z|^2})^{p/2}\leq\|F_n\|^{p}_{\cF^2}=1$ by Proposition \ref{pro1}, we see that the sequence $F_n$ is equibounded on the compact subsets of $\bC$. Hence there is a subsequence, that we continue to call $F_n$,  uniformly converging on the compact subsets to a holomorphic function $F$.  By the Fatou theorem, $F\in\cF^2(\bC)$ and $\|F\|_{\cF^2}\leq 1$. Now, since $\Omega$ has finite measure, for every $\epsilon>0$ there exists a compact subset $K\subset\bC$ such that $|\Omega\setminus K|<\epsilon$, so that $\int_{\Omega\setminus K} u_n<\epsilon$ and $\int_{\Omega\setminus K} |F(z)|^p e^{-p\pi|z|^2/2}\, dz<\epsilon$. Together with the already mentioned convergence on the compact subsets, this implies that $\int_{\Omega} u_n(z)\,dz\to \int_{\Omega} |F(z)|^p e^{-p\pi|z|^2/2}\, dz$.  As a consequence, $F\not=0$ and, since $\|F\|_{\cF^2}\leq 1=\|F_n\|_{\cF^2}$,
\[
\lim_{n\to \infty}\frac{\int_\Omega |F_n(z)|^p e^{-p\pi|z|^2/2} }{\|F_n\|^p_{\cF^2}} \leq \frac{ \int_{\Omega} |F(z)|^p e^{-p\pi|z|^2/2}\, dz}{\|F\|^p_{\cF^2}}.
\]
The reverse inequality is obvious, because $F_n$ is a maximizing sequence. 
\end{proof}
\begin{theorem}[Local Lieb's uncertainty inequality for the STFT]\label{thm locallieb} Let $2\leq p<\infty$. For every measurable subset $\Omega\subset\bR^2$ of finite measure, and every $f\in\ L^2(\bR)\setminus\{0\}$,
\begin{equation}\label{eq locallieb}
\frac{\int_\Omega |\cV f(x,\omega)|^p\, dxd\omega}{\|f\|^p_{L^2}}\leq\frac{2}{p}\Big(1-e^{-p|\Omega|/2}\Big).
\end{equation}
Moreover,  equality occurs (for some $f$ and for some $\Omega$ such that
$0<|\Omega|<\infty$) if and only if 
$\Omega$ is equivalent, 
up to a set of measure zero, to
a ball centered at some $(x_0,\omega_0)\in\bR^{2}$, and
\begin{equation*}
f(x)=ce^{2\pi ix \omega_0}\varphi(x-x_0),\qquad c\in\bC\setminus\{0\},
\end{equation*}
where $\varphi$ is the Gaussian in \eqref{defvarphi}.
\end{theorem}
Observe that when $p=2$ this result reduces to Theorem \ref{thm mainthm}. Moreover, by monotone convergence, from \eqref{eq locallieb} we obtain
\begin{equation}\label{eq liebineq}
\int_{\bR^2} |\cV f(x,\omega)|^p\, dxd\omega\leq \frac{2}{p}\|f\|^p_{L^2}, \quad f\in L^2(\bR),
\end{equation}
which is Lieb's inequality for the STFT with a Gaussian window (see \cite{lieb} and \cite[Theorem 3.3.2]{grochenig-book}). Actually, \eqref{eq liebineq} will be an ingredient of the proof of Theorem \ref{thm locallieb}.
\begin{proof}[Proof of Theorem \ref{thm locallieb}]  Transfering the problem in the Fock space $\cF^2(\bC)$, it is sufficient to prove that 
\[
\frac{\int_\Omega |F(z)|^p e^{-p\pi|z|^2/2}\, dz}{\|F\|^p_{\cF^2}}\leq \frac{2}{p}\Big(1-e^{-p|\Omega|/2}\Big)
\]
for $F\in \cF^2(\bC)\setminus\{0\}$, $0<|\Omega|<\infty$, 
and that the extremals are given by the functions $F=cF_{z_0}$ in \eqref{eq Fz0}, together with the balls $\Omega$ of center $z_0$. We give only a sketch of the proof, since the argument is similar to the proof of Theorem \ref{thm36}.  \par
Assuming $\|F\|_{\cF^2}=1$ and setting $
u(z)= |F(z)|^p e^{-p\pi|z|^2/2}
$,
arguing as in the proof of Proposition \ref{prop34} we obtain that
\[
\left(\int_{\{u=u^*(s)\}} |\nabla u|^{-1} \dH\right)^{-1}
\leq \frac{p}{2}u^*(s)\qquad\text{for a.e. $s>0$,}
\]
which implies $(u^*)'(s)+ \frac{p}{2} u^*(s)\geq 0$ for a.e.\ $s\geq 0$. With the notation of Lemma \ref{lemma3.3} we obtain $I''(s)+ \frac{p}{2} I'(s)\geq 0$ for a.e.\ $s\geq 0$, i.e. $e^{sp/2}I'(s)$ is non decreasing on $[0,+\infty)$. Arguing as in the proof of Theorem \ref{thm36} we reduce ourselves to study the inequality $I(s)\leq \frac{2}{p}(1-e^{-ps/2})$ or equivalently, changing variable $s= -\frac{2}{p}\log \sigma$, $\sigma\in (0,1]$,
\begin{equation}\label{eq gsigma2} 
G(\sigma):=I\Big(-\frac{2}{p}\log \sigma\Big)\leq \frac{2}{p}(1-\sigma)\qquad \forall\sigma\in (0,1].
\end{equation}
We can prove this inequality and discuss the case of strict inequality as in the proof of Theorem \ref{thm36}, the only difference being that now $G(0):=\lim_{\sigma\to 0^+} G(\sigma)=\int_{\bR^2} u(z)\, dz\leq 2/p$ by \eqref{eq liebineq} (hence, at $\sigma=0$ strict inequality may occur 
in \eqref{eq gsigma2}, but this is enough) and, when in \eqref{eq gsigma2} the equality occurs for some
(and hence for every) $\sigma\in[0,1]$, in place of \eqref{catena} we will have 
\begin{align*}
1=\max u =\max |F(z)|^p e^{-p\pi |z|^2/2}&= (\max |F(z)|^2 e^{-\pi |z|^2})^{p/2} \\
&\leq \|F\|^p_{\cF^2}=1.
\end{align*}
The ``if part" follows by a direct computation. 
\end{proof} 
\subsection{$L^p$-concentration estimates for the STFT}
 We consider now the problem of the concentration in the time-frequency plane
 in the $L^p$ sense, which has interesting applications to signal recovery (see e.g. \cite{abreu2019}). More precisely, Theorem \ref{thm lpconc} below proves a conjecture of
Abreu and Speckbacher \cite[Conjecture 1]{abreu2018} (note that when $p=2$ one obtains Theorem \ref{thm mainthm}).

Let $\cS'(\bR)$ be the space of temperate distributions and, for $p\geq 1$, consider the subspace (called \textit{modulation space} in the literature \cite{grochenig-book}) 
\[
M^p(\bR):=\{f\in\cS'(\bR): \|f\|_{M^p}:=\|\cV f\|_{L^p(\bR^2)}<\infty\}.
\]
In fact, the definition of STFT (with Gaussian or more generally Schwa\-rtz window) and the Bargmann transform $\cB$ extend (in an obvious way) to injective operators on $\cS'(\bR)$. It is clear that $\cV: M^p(\bR)\to L^p(\bR^2)$ is an isometry, and it can be proved that $\cB$ is a \textit{surjective} isometry from $M^p(\bR)$ onto the space $\cF^p(\bC)$ of holomorphic functions $F(z)$ satisfying 
\[
\|F\|_{\cF^p}:=\Big(\int_\bC |F(z)|^p e^{-p\pi|z|^2/2}\, dz\Big)^{1/p}<\infty,
\]
 see e.g. \cite{toft} for a comprehensive discussion. Moreover the formula \eqref{eq STFTbar} continues to hold for $f\in\cS'(\bR)$.

 \begin{theorem}[$L^p$-concentration estimates for the STFT]\label{thm lpconc} Let $1\leq p<\infty$. For every measurable subset $\Omega\subset\bR^2$ of finite measure and every $f\in M^p(\bR)\setminus\{0\}$,
\begin{equation}\label{eq lpconc}
\frac{\int_\Omega |\cV f(x,\omega)|^p\, dxd\omega}{\int_{\bR^2} |\cV f(x,\omega)|^p\, dxd\omega}\leq 1-e^{-p|\Omega|/2}.
\end{equation}
Moreover,  equality occurs (for some $f$ and for some $\Omega$ such that
$0<|\Omega|<\infty$) if and only if 
$\Omega$ is equivalent, 
up to a set of measure zero, to
a ball centered at some $(x_0,\omega_0)\in\bR^{2}$, and
\begin{equation}\label{eq lp concert optimal} 
f(x)=ce^{2\pi ix \omega_0}\varphi(x-x_0),\qquad c\in\bC\setminus\{0\},
\end{equation}
where $\varphi$ is the Gaussian in \eqref{defvarphi}.
\end{theorem}
We omit the proof, that is very similar to that of Theorem \ref{thm36}.
We just observe that \eqref{eq bound} extends to any $p\in [1,\infty)$ as 
\[
|F(z)|^p e^{-p\pi|z|^2/2}\leq \frac{p}{2} \|F\|^p_{\cF^p}
\] and again the equality occurs at some point $z_0\in\bC$ if and only if $F=cF_{z_0}$, for some $c\in\bC$ (in particular, $\|F_{z_0}\|^p_{\cF^p}=2/p$); see  e.g. \cite[Theorem 2.7]{zhu}.

 
 As a consequence we obtain at once the following sharp uncertainty principle for the STFT.
 \begin{corollary}\label{cor corsez5}
Let $1\leq p<\infty$. If for some $\epsilon\in (0,1)$, some function $f\in M^p(\bR)\setminus\{0\}$ 
 and some $\Omega\subset\bR^2$ we have $\int_\Omega|\cV f(x,\omega)|^p \, dxd\omega\geq (1-\epsilon)\|f\|^p_{M^p}$, then necessarily
\begin{equation}\label{eq stima eps cor}
|\Omega|\geq \frac{2}{p}\log(1/\epsilon),
\end{equation}
with equality if and only if $\Omega$ is a ball and $f$ has the form \eqref{eq lp concert optimal}, where $\varphi$ is the Gaussian in \eqref{defvarphi}
and $(x_0,\omega_0)$ is the center of the ball. 
\end{corollary}
We point out that, in the case $p=1$, the following rougher --but valid for any window in $M^1(\bR)\setminus\{0\}$-- lower bound
\[
|\Omega|\geq 4(1-\epsilon)^2
\]
 was obtained in \cite[Proposition 2.5.2]{grochenig}. Arguing as in Section \ref{sec mult} it would not be difficult to suitably generalise Theorem \ref{thm lpconc} and Corollary \ref{cor corsez5} in arbitrary dimension.

\subsection{Changing window}\label{sec change wind}

Theorem \ref{thm mult} can be suitably reformulated when the Gaussian window $\varphi$ in \eqref{eq gaussian dimd} is dilated or, more generally, replaced by $\mu(\cA)\varphi$, where $\mu(\cA)$ is a metaplectic operator associated with a symplectic matrix $\cA\in Sp(d,\bR)$ (recall that, in dimension 1, $Sp(1,\bR)=SL(2,\bR)$ is the special linear group of $2\times 2$ real matrices with determinant $1$). 

We address to \cite[Section 9.4]{grochenig} for a detailed introduction to the metaplectic representation. Roughly speaking one associates, with any matrix $\cA\in Sp(d,\bR)$, a unitary operator $\mu(\cA)$ on $L^2(\bR^d)$ defined up to a phase factor, providing a projective unitary representation of $Sp(d,\bR)$ on $L^2(\bR^d)$. In more concrete terms, we know that $Sp(d,\bR)$ is generated by matrices of the type (in block-matrix notation) 
\[
\cA_1=\begin{pmatrix}
0 &I \\
-I &0
\end{pmatrix}
\qquad
\cA_2=\begin{pmatrix}
A &0 \\
0 &{A^*}^{-1}
\end{pmatrix}
\qquad
\cA_3=\begin{pmatrix}
I &0 \\
C & I
\end{pmatrix}
\]
where $A\in GL(d,\bR)$, and $C$ is real and symmetric ($I$ denoting the identity matrix). The corresponding operators are then given by $\mu(\cA_1)=\cF$ (Fourier transform), $\mu(\cA_2)f(x)=|{\rm det}\,A|^{-1/2}f(A^{-1}x)$ and $\mu(\cA_3)f(x)=e^{\pi iC x\cdot x} f(x)$ (up to a phase factor). \par
Now, the relevant property of the STFT  is its symplectic covariance (see \cite[Lemma 9.4.3]{grochenig-book}):
\[
|\cV_{\mu(\cA)\varphi} (\mu(\cA)f)(x,\omega)|=|\cV_{\varphi} (f)(\cA^{-1}(x,\omega))|.
\]
As a consequence, if we define, for $g,f\in L^2(\bR^d)\setminus\{0\}$, the quotients
\[
\Phi_{\Omega,g}(f):=\frac{\int_\Omega |\cV_g f(x,\omega)|^2\, dxd\omega}{\int_{\bR^{2d}} |\cV_g f(x,\omega)|^2\, dxd\omega},
\] 
we obtain (since ${\rm det}\,\cA=1$)
\[
\Phi_{\Omega,\mu(\cA)\varphi}(\mu(\cA)f)=\Phi_{\cA^{-1}(\Omega),\varphi}(f).
\]
Hence, since $\cA$ is measure preserving and $\mu(\cA)$ is a unitary operator, we deduce at once from Theorem \ref{thm mult} that for every measurable subset $\Omega\subset\bR^{2d}$ of finite measure and every $f\in L^2(\bR^d)\setminus\{0\}$,
\[
\frac{\int_\Omega |\cV_{\mu(\cA)\varphi} f(x,\omega)|^2\, dxd\omega}{\|f\|_{L^2}^2}\leq  \frac{\gamma(d,c_\Omega)}{(d-1)!}
\]
with $c_\Omega=\pi(|\Omega|/\boldsymbol{\omega}_{2d})^{1/d}$. Moreover, the equality occurs if and only if $f=\mu(\cA)M_{\omega_0}T_{x_0}\varphi$ (recall \eqref{eq gaussian dimd}) and $\Omega$ is equivalent, up to a set of measure zero, to $\cA(B)$ for some ball $B\subset\bR^{2d}$ centered at $(x_0,\omega_0)$,
where  $T_{x_0}f(x)=f(x-x_0)$ and $M_{\omega_0}f(x)=e^{2\pi ix\cdot\omega_0}f(x)$.
\par\bigskip
\nopagebreak
\noindent{\bf Acknowledgements.} We wish to thank Nicola Fusco for useful discussion on the validity of
Lemma \ref{lemmau*}, and for addressing us to some relevant references.

\section*{Appendix}

\begin{figure}[H]
\includegraphics[width=12cm, height = 7cm]{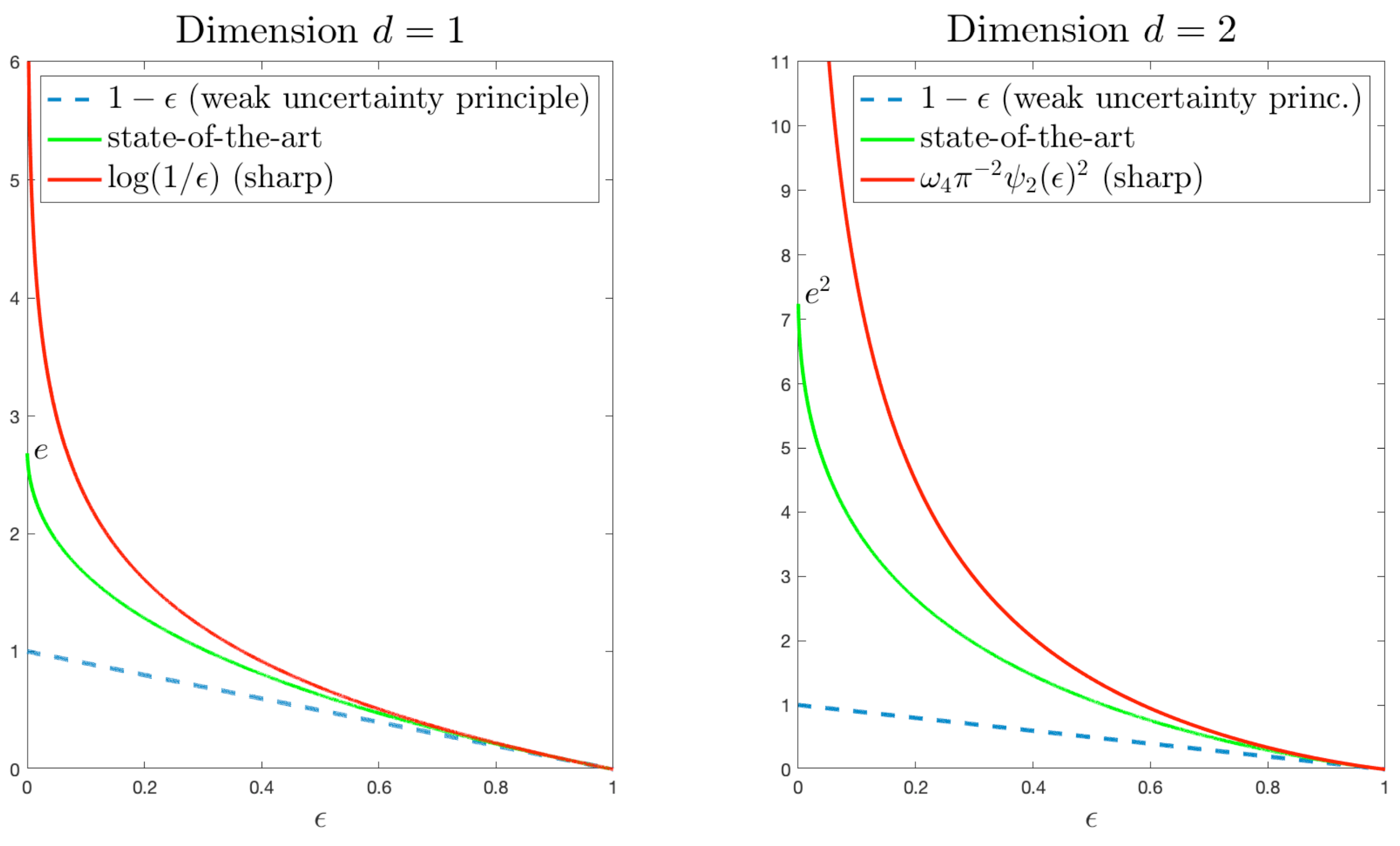}
\caption{
Left: in dimension 1, assuming that $\Omega\subset\bR^2$  captures a fraction $1-\epsilon$ of the energy of some function $f\in L^2(\bR)$, comparison between the lower bound for $|\Omega|$ in \eqref{eq statart} (state-of-the-art), the sharp lower bound $\log(1/\epsilon)$ in \eqref{eq stima eps} and the so-called \textit{weak uncertainty principle} $|\Omega|\geq 1-\epsilon$ \cite[Proposition 3.3.1]{grochenig-book} (which follows at once from the elementary estimate $\|\cV f\|_{L^\infty}\leq\|f\|_{L^2}$).\newline
Right: the same comparison in dimension $d=2$. Here the state-of-the-art is represented by \eqref{bound groc dim d}, whereas the sharp bound $|\Omega|\geq \boldsymbol{\omega}_4 \pi^{-2}\psi_2(\epsilon)^2$ is given in \eqref{uncertainty dim d}  ($\boldsymbol{\omega}_4=\pi^2/2$).}\label{figure1}
\vspace{4mm}
\includegraphics[width=12cm, height = 6.5cm]{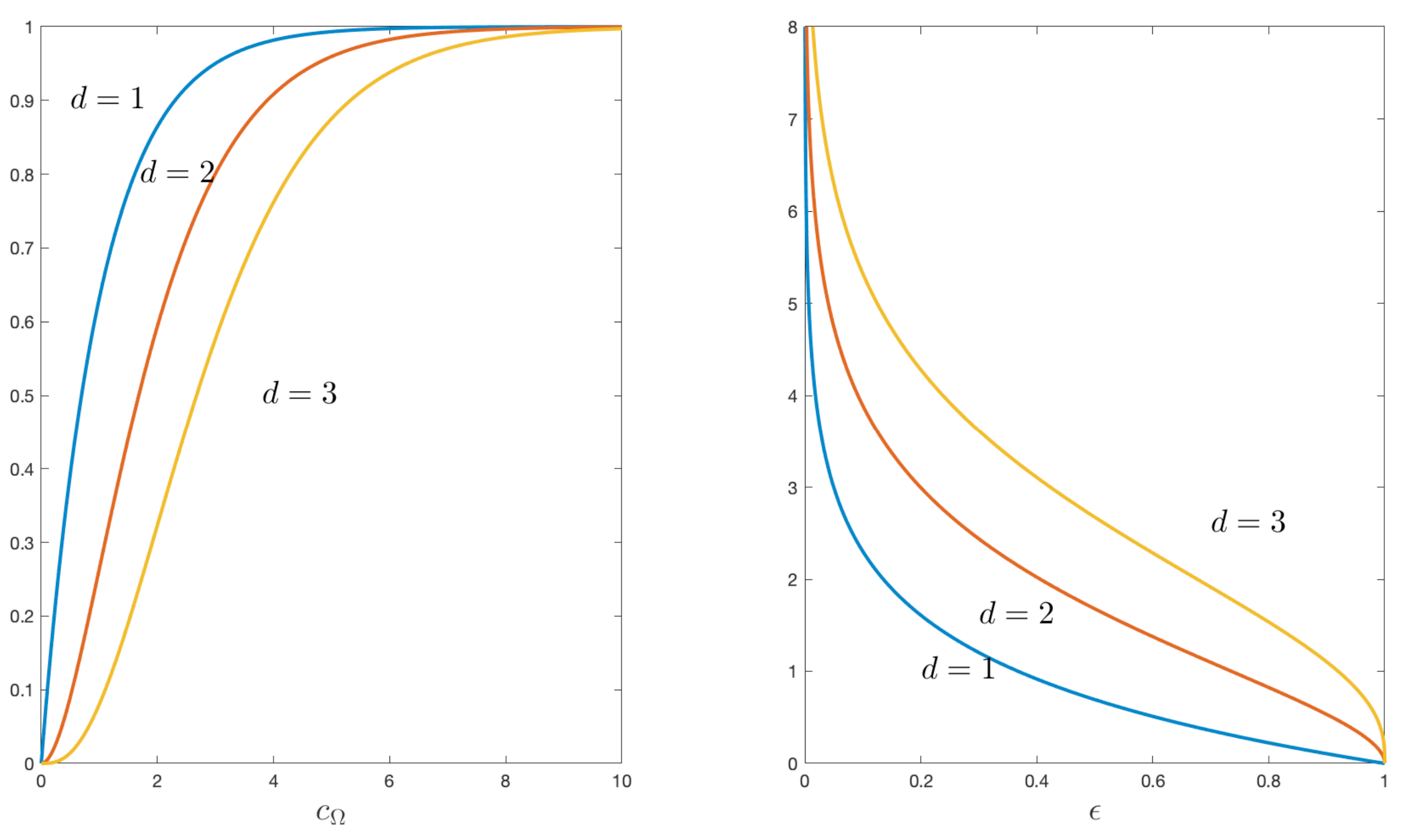}
\caption{Left: The upper bound $\gamma(d,c_\Omega)/(d-1)!$ in \eqref{eq thm mult}, for $d=1,2,3$, as a function of $c_\Omega=\pi(|\Omega|/\boldsymbol{\omega}_{2d})^{1/d}$.\newline
Right: The lower bound $\psi_d(\epsilon)$ for $c_\Omega$ in \eqref{uncertainty dim d}, for $d=1,2,3$. Recall, $\psi_d(\epsilon)$ is the inverse function of $1-\gamma(d,s)/(d-1)!$, in particular $\psi_1(\epsilon)=\log(1/\epsilon)$.} \label{figure2}
\end{figure}


\newpage

\bibliographystyle{abbrv}
\bibliography{biblio.bib}

\end{document}